\documentclass[reqno]{amsart}
\usepackage{amsthm,amsmath,amssymb,mathrsfs,graphicx,color,url}

\theoremstyle{plain}
\newtheorem{thm}{Theorem}[section]
\newtheorem{prop}[thm]{Proposition}
\newtheorem{lem}[thm]{Lemma}

\theoremstyle{definition}

\newtheorem{ex}[thm]{Example}

\theoremstyle{remark}

\newtheorem{rem}[thm]{Remark}

\numberwithin{equation}{section}
\def\hyp{\mathrm{hyp}}
\def\hull{\mathrm{span}}
\def\fphi{f_{\phi}}
\def\fphihat{\hat{f}_{\phi}}
\def\fphicheck{\check{f}_{\phi}}
\def\ghat{\hat{g}}
\def\gcheck{\check{g}}
\def\fcheck{\check{f}}
\def\one{\mathbf{1}} 
\def\fsp{\cH}
\def\fbsp{\cH_b}
\def\zbar{\bar{z}}
\def\wbar{\bar{w}}
\def\disp{\displaystyle} 
\def\fhat{{\widehat{f}}}
\def\bra{\langle}
\def\ket{\rangle}
\def\as{a.s.}
\def\phi{\varphi} 
\def\conv{\mathrm{conv}}
\def\refl{\mathrm{ref}}

\newcommand{\newoperator}[2]{\DeclareMathOperator{#1}{#2}}
\newoperator{\supp}{supp}
\newoperator{\tr}{Tr}
\newoperator{\re}{Re}
\newoperator{\im}{Im}
\newoperator{\sgn}{sgn}
\newoperator{\pf}{Pf}
\newoperator{\per}{per}
\newoperator{\var}{var}

\def\Z{\mathbb{Z}}
\def\E{\mathbb{E}}
\def\C{\mathbb{C}}
\def\D{\mathbb{D}}
\def\R{\mathbb{R}}
\def\P{\mathbb{P}}
\def\cZ{\mathcal{Z}}
\def\cU{\mathcal{U}}
\def\cV{\mathcal{V}}
\def\cK{\mathcal{K}}
\def\cH{\mathcal{H}}
\def\cI{\mathcal{I}}
\def\cP{\mathcal{P}}
\def\cQ{\mathcal{Q}}
\def\cR{\mathcal{R}}
\def\cS{\mathcal{S}}

\begin{document} 
\title[The density of zeros of random power series]{The
density of zeros of random power series with stationary complex Gaussian coefficients} 
\author[T.~Shirai]{Tomoyuki SHIRAI}
\address{Institute of Mathematics for Industry, Kyushu University, 744 Motooka, Nishi-ku, Fukuoka 819-0001, Japan}
\email{shirai@imi.kyushu-u.ac.jp} 
\begin{abstract} 
We study the zeros of random power series with
 stationary complex Gaussian coefficients, 
whose spectral measure is absolutely continuous. 
We analyze the precise asymptotic behavior of the radial density of zeros near the boundary of the circle of convergence. The dependence of the coefficients
 generally reduces the density of zeros compared
 with that of the hyperbolic Gaussian analytic function (the
 i.i.d. coefficients case), 
where the spectral density and its zeros plays a crucial
 role in this reduction. 
We also show the relationship between the support of the
 spectral measure and the analytic continuation at the
 boundary of the circle of convergence.
\end{abstract}
\maketitle 

\section{Introduction}  \label{sec:intro}

Gaussian analytic functions have been the subject of 
extensive research from various points of view over the years
(cf. \cite{Kahane85,Sodin00,Sodin-Tsirelson-04,
Peres-Virag05, K09,
HKPV09,N12,NS12,Shirai12,Feldheim13,ANS17,GP17,BNPS18,Feldheim-Dvora18,B21,KN21,BN22,HKR22} and
references therein). 
The case of random power series with i.i.d. coefficients has been 
extensively studied. In this paper, we focus on 
Gaussian power series with dependent coefficients under the
following setting. 
Let $\Xi = \{\xi_k\}_{k \in \Z}$ be a stationary, centered,
complex Gaussian process with covariance  
\[
 \gamma(k-l) = \E[\xi_k \overline{\xi_l}], 
\]
where we always assume that $\gamma(0)=1$, i.e., the
distribution of each marginal $\xi_k$ is 
complex standard normal. 
We consider the Gaussian power series with dependent coefficients $\Xi=\{\xi_k\}_{k \in \Z}$ 
\[
 X_{\Xi}(z) := \sum_{k=0}^{\infty} \xi_k z^k
\]
and study its random zeros. The radius of convergence of
$X_{\Xi}(z)$ is equal to $1$ a.s. 
since $\lim_{k \to \infty} |\xi_k|^{1/k}=1$ a.s. as long as
$\gamma(0) = \E[|\xi_k|^2]=1$ (cf. Lemma~\ref{lem:normal_sequence}.)
The zeros of $X_{\Xi}(z)$ form a point process
$\cZ_{X_{\Xi}} := \sum_{z \in \D : X_{\Xi}(z) =0} \delta_z$ on $\D:=\{z \in \C : |z| < 1\}$. 
When $\gamma(k) = \delta_{k0}$, i.e., $\Xi = \{\xi_k\}_{k
\in \Z}$ is an i.i.d. sequence, the Gaussian analytic
function (GAF) is called the hyperbolic GAF, 
and its zero process is known to be the determinantal
point process associated with the Bergman kernel $K_{\mathrm{Berg}}(z,w) =
(1-z\wbar)^{-2}$ and the background measure $\pi^{-1} m(dz)$
\cite{Peres-Virag05} (see for details on DPPs (cf. \cite{ST00, So00, ST03a, HKPV09}). From this fact, it is easily seen that
the expected number of zeros of the hyperbolic GAF within $\D_r = \{z \in \C : |z|<r\}$ 
is exactly equal to $r^2(1-r^2)^{-1}$. 
In \cite{Mukeru20} the authors treated
the case where the coefficient Gaussian process $\Xi$ corrsponds to a fractional
Gaussian noise with Hurst index $0 \le H < 1$ and gave an
estimate for $\E[\cZ_{X_{\Xi}}(\D_r)]$, showing that it is
small by $O((1-r^2)^{-1/2})$ compared with the hyperbolic GAF case. 
In \cite{Noda-Shirai23}, the case where $\Xi$ is finitely dependent, i.e., 
$\gamma(k)$ is a trigonometric polynomial, was considered, and
the precise asymptotics of the expected number of zeros were
analyzed. It is also shown that the dependence of the
coefficients always reduces the expected number of 
zeros. Furthermore, the spectral density and its zeros of the coefficient Gaussian
process $\Xi$, in particular, the degeneracy of zeros, 
plays a crucial role in determining the extent of this reduction. 
With the results of \cite{Noda-Shirai23} in mind, 
this paper discusses the precise asymptotic behavior of the 
radial density of zeros as it approaches 
the boundary of the circle of convergence, i.e., $\partial \D$. 
Moreover, we discuss the analytic continuation of
$X_{\Xi}(z)$ in relation to the strength of the dependence
among the coefficients in $\Xi$. Intuitively, we will see
that the stronger the dependence of $\Xi$, the easier the analytic continuation
becomes. These observations  
 provide deeper insights into the relationship between the
 spectral measure of the coefficients and the distributions of zeros. 

The rest of the paper is organized as follows: 
In Section~\ref{sec:main}, we present the main results of the paper. 
Section~\ref{sec:kostlan} discusses the spectral representation of the
coefficinet Gaussian process $\Xi$ and provides a spectral representation of the
density of zeros of $X_{\Xi}(z)$. In Section~\ref{sec:poisson_integral}, we refine the asymptotic
behavior of the Poisson integral and its variant near the
boundary of the unit disk. Section~\ref{sec:proof} contains
the proof of the main Theorem~\ref{thm:main}. Finally, in
Section~\ref{sec:analytic_continuation}, we discuss the
analytic continuation of $X_{\Xi}(z)$ in terms of the spectral
measure of $\Xi$.

\section{Main results} \label{sec:main}

The variance of $X_{\Xi}(z)$ in terms of
the spectral measure $F(dt)$ of the coefficient Gaussian
process $\Xi$ is given by 
\[
K(z,z) := \E[|X_{\Xi}(z)|^2] 
= \frac{1}{1-r^2} \int_{-\pi}^{\pi} P_r(\phi-t)
F(dt),  
\] 
where $z = r e^{i\phi}$ and $P_r(s)$ is the Poisson kernel 
\[
P_r(s) = \frac{1-r^2}{1-2r \cos s + r^2} \quad \text{for $s
\in [-\pi,\pi]$ and $0 \le r < 1$}. 
\]
The covariance kernel of $X_{\Xi}(z)$ will be given in \eqref{eq:Kzw}. 
We often identified $e^{i\phi}$ with $\phi \in [-\pi,\pi]$ and 
we write $f(\phi)$ for $f(e^{i\phi})$. 
If $F(dt)$ is absolutely continuous, i.e., $F(dt) = f(t)dt$,
then we have the following representation: 
\[
 K(z,z) = \frac{1}{1-|z|^2} \E_z[f(B_{\tau})], 
\]
where $B_t$ is the complex Brownian motion starting at $z
\in \D$ and $\tau$ is the first hitting time to the boundary
$\partial \D$. Thus, as $z$ approaches the boundary $e^{i \phi}$, 
i.e., $|z|$ approaches $1$, the variance $K(z,z)$ diverges and $X_{\Xi}(z)$
oscillates very fast (and it tends to have more zeros) unless $f(e^{i\phi})$ vanishes. 
For example, it is well-known that a Gaussian process $\Xi = \{\xi_k\}_{k
\in \Z}$ is \textit{purely non-deterministic} if and only if its
spectral measure $F(ds)$ is absolutely continuous, i.e.,
$F(ds) = f(s) ds$ such that 
$\int_{-\pi}^{\pi} \log f(s) ds > -\infty$ \cite[p.112]{Ibragimov-Rozanov78}. 
In this case, $f(s) > 0$ a.e. on $\partial \D$, and hence
$X_{\Xi}(z)$ oscillates infinitely often near the boundary
$\partial \D$, and the boundary $\partial \D$ becomes the
natural boundary (see Section~\ref{sec:analytic_continuation}). 

In order to state the main result, 
we introduce some notations. 
The symmetrization and anti-symmetrization of $h$
are defined by 
\begin{equation}
 \hat{h}(s) = \frac{h(s) + h(-s)}{2}, \quad 
 \check{h}(s) = \frac{h(s) - h(-s)}{2}. 
\label{eq:sym-antisym} 
\end{equation}
We define the operator $T$ acting on symmetric functions smooth at
$s=0$ by 
\[
Th(s) := 
\begin{cases}
\disp \frac{h(s) - h(0)}{1-\cos s} & \text{for $s \not= 0$}, \\[3mm]
h''(0) & \text{for $s=0$}. 
\end{cases}
\]
and $\cI$ is the integral operator defined by 
\[
 \cI(h) = \frac{1}{2\pi} \int_{-\pi}^{\pi} h(s) ds. 
\]

The following is the first main theorem of this paper. 

\begin{thm}\label{thm:main}
Let $F_{\phi}(ds) = F(d(s+\phi))$ on $(-\pi, \pi]$ modulo $2\pi$. 
Suppose $F_{\phi}(ds) =  f_{\phi}(s)ds$ and $f_{\phi}(s)$ is smooth at $s=0$. \\
\noindent
(i) When $f_{\phi}(0) > 0$, 
\begin{align*}
  \rho_1(re^{i\phi})
&=\frac{1}{\pi(1-r^2)^2} 
- \frac{1}{4f_{\phi}(0)^2} \Big\{ f_{\phi}'(0)^2 + \cI(T \fphihat)^2 \Big\}
+ O(1-r^2). 
\end{align*}
(ii) When $f_{\phi}(0) = 0$, 
\begin{align*}
  \rho_1(re^{i\phi})
&=\frac{1}{\pi(1-r^2)} 
\frac{f_{\phi}''(0)}{2\cI(T\fphihat)} + O(1)
\end{align*}
(iii) When $\fphi(0)=\fphi''(0)=0$. Then, 
\begin{align*}
\rho_1(re^{i\phi})
= \frac{1}{\pi} \frac{\cI(T^2 \fphihat)^2 -
 \cI(T^2(\fphihat \cos s))^2 
- \cI\big(T^2(\fphicheck \sin s)\big)^2} 
{4 \cI(T\fphihat)^2} + O(1-r^2). 
\end{align*}
\end{thm}
When the spectral density vanishes at some points on the unit circle, 
it is observed that the order of the density of zeros in the
direction approaching those points decreases according to
the degree of the degeneracy. 
This is related to the phenomenon analyzed in
\cite{Noda-Shirai23}, where the expected number of zeros is
further reduced when the spectral density has zeros. 
The opposite effect was observed in \cite{APP22}, where the expected number of real zeros of 
Gaussian trigonometric polynomials increases due to the
degeneracy of the spectral measure.  

\begin{rem}\label{rem:integral}
Since $\fphihat(s)$ is of $O(s^2)$ at $s=0$ when $\fphi(0)=0$, we have 
\[
 \cI(T\fphihat) 
= \frac{1}{2\pi}\int_{-\pi}^{\pi}
 \frac{\fphihat(s)}{1-\cos s} ds. 
\]
Note that $\fphihat(s), \fphihat(s) \cos s, \fphicheck(s) \sin s$ are
 of $O(s^4)$ when $\fphi(0)=\fphi''(0)=0$. 
For such a function $g$ of $O(s^4)$ at $s=0$, we have 
\[
\cI(T^2 g) 
= \frac{1}{2\pi}\int_{-\pi}^{\pi} \frac{g(s)}{(1-\cos s)^2} ds. 
\]
Then, the numerator of the first term in (iii) can be expressed as 
\[
 \frac{1}{4} \int_{-\pi}^{\pi}\int_{-\pi}^{\pi} 
\frac{G_{\phi}(s,t) + G_{\phi}(s,-t)}{(1-\cos s)^2 (1-\cos
 t)^2} dsdt, 
\]
where 
$G_{\phi}(s,t) :=
 \{f_{\phi}(s)f_{\phi}(t)+f_{\phi}(-s)f_{\phi}(-t)\}\{1-\cos(s-t)\}$,
 and it is obviously positive. 
\end{rem}

\begin{ex}[$1$-dependent case]
We consider the case where the coefficient Gaussian process $\Xi$ is
 $1$-dependent, i.e., the covariance function is given by 
\[
 \gamma(k) = 
\begin{cases}
 1 & k=0, \\
 a & |k|= 1, \\
 0 & |k| \ge 2 
\end{cases}
\]
with $|a| \le 1/2$.  
In this case, 
$f_{\phi}(s) = 1 + 2a \cos (s+\phi)$. 
Then, 
\[
\fphi(0) = 1+2a \cos \phi, \ 
\fphi'(0) = -2a \sin \phi, \ 
\fphi''(0) = -2a \cos \phi,  
\]
and 
\[
 \fphihat(s) = 1 + 2a \cos \phi \cos s, \quad  
T \fphihat(s) 
=\frac{\fphihat(s) - \fphihat(0)}{1-\cos s} 
= -2a \cos \phi. 
\]
Note that $\fphi(0)=0$ only if 
$(a,\phi)=(1/2, \pi)$ or $(a,\phi)=(-1/2, 0)$. \\
\noindent 
(i) When $|a| \le 1/2$ and $(a,\phi)\not= (1/2, \pi), (-1/2, 0)$, 
since $\cI(T \fphihat) = -2a \cos \phi$, we have 
\[
R(f_{\phi}) 
:= \frac{1}{4 f_{\phi}(0)^2} 
\Big(f_{\phi}'(0)^2 + \cI(T\fphihat)^2 \Big)
= \frac{a^2}{(1+2a \cos \phi)^2},  
\]
and then   
\[
 \rho_1(re^{i\phi}) = \frac{1}{\pi(1-r^2)^2} - 
\frac{a^2}{(1+2a \cos \phi)^2} 
+ O(1-r^2). 
\]
(ii) When $(a,\phi)=(1/2, \pi)$ or $(a,\phi)=(-1/2, 0)$, we
 have 
\[
 \rho_1(re^{i\phi}) 
= \frac{1}{\pi(1-r^2)} \frac{f_{\phi}''(0)}{2\cI(T\fphihat)} + O(1)
= \frac{1}{2\pi(1-r^2)} + O(1). 
\]
\end{ex}

\begin{ex} \label{ex:half-interval}
We consider the case $F(ds) = I_{[-\pi/2,\pi/2]}(s)ds$ and
 let $\{\xi_k\}_{k \in \Z}$ be the corresponding Gaussian process. 
The right panel of Fig.~\ref{fig:fig1} shows the zeros of
 the approximate polynomial of $X_F(z) = \sum_{k=0}^{\infty}
 \xi_k z^k$. The zeros in the right half are distributed
 almost the same as in the left panel, while the zeros in
 the left half are neatly aligned near the unit
 circle. These correspond to the zeros that might  disappear
 in the limit as $N \to \infty$, which is  consistent with
 the fact that the density of the zeros on  the left is $O(1)$.

\begin{figure}[htbp]
\begin{center}
\includegraphics[width=0.9 \hsize]{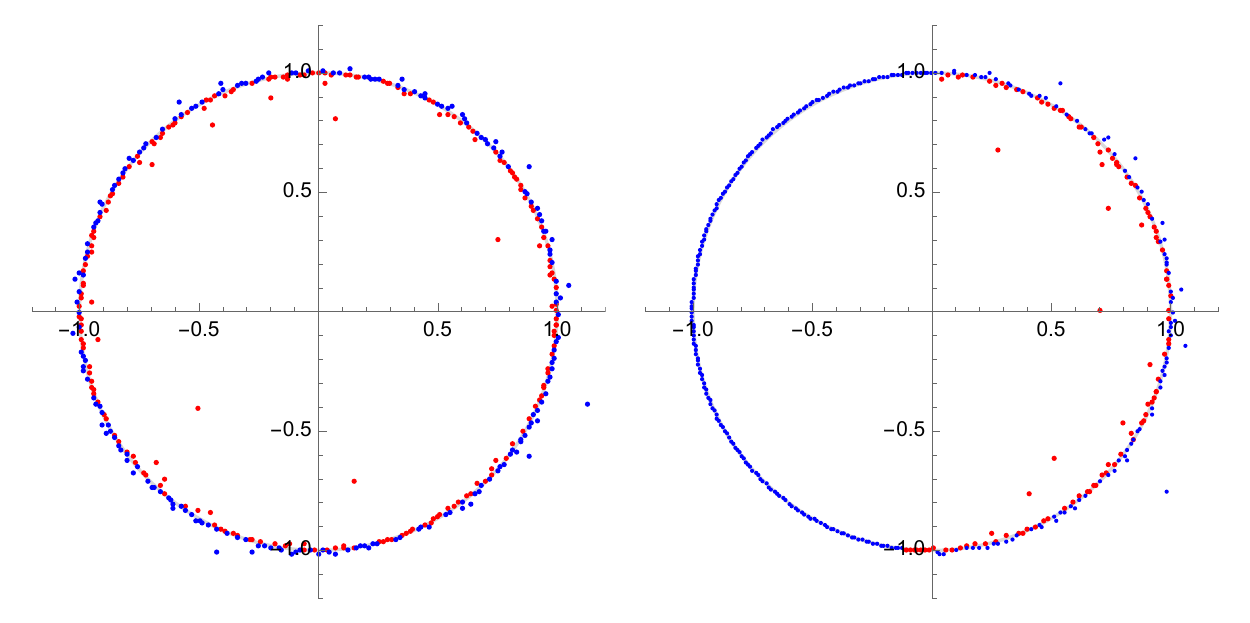}
\end{center}
\caption{Left: The zeros of the approximate polynomial $X_{\hyp}^{(400)}(z) =
 \sum_{k=0}^{400} \xi_k z^k$ for $F(dt) = dt/(2 \pi)$,
 the i.i.d. case. Red points indicate zeros
 inside the unit disk, while blue points indicate zeros
 outside the unit disk. 
Right: The zeros of an approximate polynomial of $X_F(z)$ for
 $F(dt)=\one_{[-\pi/2,\pi/2]}(t)dt/\pi$.} 
\label{fig:fig1}
\end{figure}
Suppose $\pi/2 < \phi < \pi$ and 
\[
 f_{\phi}(s) 
=
\begin{cases}
1 & -\pi < s < -(\phi -\pi/2), \\ 
0 & -(\phi -\pi/2) < s < 3\pi/2-\phi, \\ 
1 & 3\pi/2-\phi < s < \pi. 
\end{cases}
\]
Then, we see that 
\[
 \fhat_{\phi}(s) 
=
\begin{cases}
0 & |s| < \phi - \pi/2, \\ 
1/2 & \phi-\pi/2 < |s| < 3\pi/2-\phi, \\ 
1 & 3\pi/2-\phi < |s| < \pi,  
\end{cases}
\]
and 
\[
 \check{f}_{\phi}(s) 
=
\begin{cases}
0 & |s| < \phi - \pi/2, \\ 
-\sgn(s)/2 & \phi - \pi/2 < |s| < 3\pi/2-\phi, \\ 
0 & 3\pi/2-\phi < |s| < \pi.  
\end{cases}
\]
\begin{figure}[htbp]
\begin{center}
\includegraphics[width=1.0\hsize]{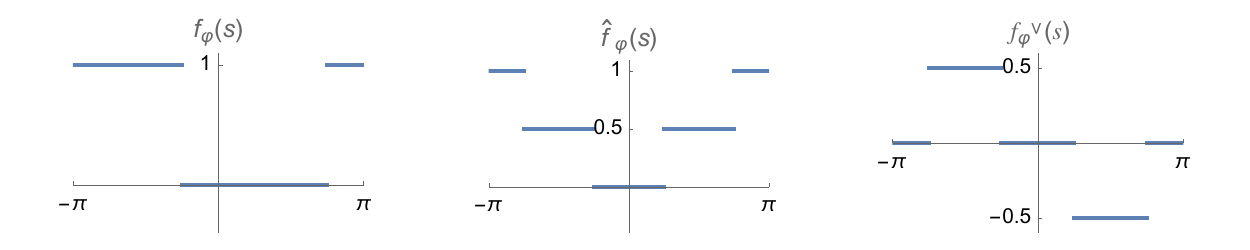}
\end{center}
\caption{The case $f(s) = I_{[-\pi/2,\pi/2]}(s)$ and
 $\phi=3\pi/4$. From left, $f_{\phi}(s)$, $\fhat_{\phi}(s)$,
 and $\fcheck_{\phi}(s)$. }
\end{figure}

In this case, $f_{\phi}(0) = f''_{\phi}(0)=0$. 
After some direct calculations for (iii) in
 Theorem~\ref{thm:main} using Remark~\ref{rem:integral}, we
 obtain the following: as $r \to 1$,  
\[
 \rho_1(re^{i\phi}) = \frac{1}{12\pi \cos^2 \phi} +
 O(1-r^2) \quad \phi \in (\pi/2, \pi]. 
\]
This shows the behavior of zeros on the left-half plane. 
\end{ex}

Finally, we discuss analytic continuation of $X_{\Xi}$. 
For a power series $h=h(z)$, a point $r e^{i \theta}$ on the
circle of convergence of $h$ is \textit{singular} if there does not exist any analytic
continuation of $h$ across an arc of $\{|z|=r\}$
containing $r e^{i\theta}$. A point on $\{|z|=r\}$ is called
\textit{regular} if it is not singular. We call the totally
of the regular points the \textit{regular set} of
$X_{\Xi}$. 
Suppose $F$ is absolutely continuous with $F(dt) = f(t) dt$. 
From Theorem~\ref{thm:main} (iii), if $f(s)=0$ on a closed
arc $I$, then the number of zeros of $X_{\Xi}$ near $I$ is of $O(1)$. This suggests
that $X_{\Xi}(z)$ could be analytically continued across
$I$. This is indeed the case, and the support of the spectral
measure remains crucial even when it is not absolutely continuous. 

\begin{thm}\label{thm:analytic_continuation}
Let $\supp F$ be the support of the spectral measure $F(dt)$
 of $\Xi$. 
Then $(\supp F)^c$ is the regular set of $X_{\Xi}$, i.e., 
$X_{\Xi}(z)$ can only be analytically extended outside the
 disk of convergence across $(\supp F)^c$. 
\end{thm}

Conceptually, stronger dependence of $\Xi$ makes analytic continuation more feasible.
The dependence of $\Xi$ is closely related to the spectral
 measure and its degeneracy. 
(i) If $\Xi$ is purely non-deterministic (see Introduction for
 the condition), then 
the circle of convergence is the natural boundary for $X_{\Xi}$.  
(ii) In Example~\ref{ex:half-interval}, 
$X_{\Xi}(z)$ can be analytically continued 
across the unit circle in the left-half plane. 
(iii) If $\Xi = (\xi_k)_{k \in \Z}$ is periodic in $k$ with
period $q$, the spectral measure is atomic. In this case,
$X_{\Xi}$ has the meromorphic extention with poles at 
$\{e^{2\pi i k/q}\}_{k=0,1,\dots,q-1}$. 
These examples are discussed in \cite{Shirai25a+}. 

In Section~\ref{sec:analytic_continuation}, we will provide a direct proof of 
the fact that the singularity of $X_{\Xi}$ coincides with
$\supp F$ (Theorem~\ref{thm:analytic_continuation}). 
Additionally, we will present an alternative proof of this result, 
which has been essentially given in \cite{ANS17} using the theory of entire
functions. This approach was pointed out by Misha Sodin. 

\section{Edelman-Kosltan's formula and spectral measures}
\label{sec:kostlan}

We consider a centered, stationary Gaussian process $\Xi=\{\xi_k\}_{k \in \Z}$ with covariance function 
$\{\gamma(k)\}_{k \in \Z}$. 
Suppose $\gamma(0)=1$. Then, by Herglotz's theorem, there exists a probability measure $F(dt)$ on 
$(-\pi, \pi]$ such that 
\[
 \gamma(k) = \int_{-\pi}^{\pi} e^{-ikt} F(dt). 
\]
The measure $F(dt)$ is called the spectral measure of Gaussian
process $\Xi$. The Gaussian process $\Xi$ with spectral measure $F(dt)$ admits
the following spectral representation 
(cf. \cite{Ibragimov-Rozanov78}): 
\begin{equation}
 \xi_k = \int_{-\pi}^{\pi} e^{-ikt} dZ(t), 
\label{eq:spec-rep-process}
\end{equation}
where $Z(t)$ is the complex Gaussian process with independent increments satisfying 
\[
 \E[Z(A) \overline{Z(B)}] = F(A \cap B). 
\]
Every element $\eta$ of the span $H(\Xi) := \hull \{\xi_k, k \in \Z\} \subset L^2(\P)$ of $\Xi = \{\xi_k\}_{k
\in \Z}$ is represented as 
\[
 \eta = \int_{-\pi}^{\pi} u(t) dZ(t). 
\]
for some $u \in L^2(F)$. Here $L^2(F)$ is the space of
square-integrable functions with respect to the inner
product 
\[
 \bra f, g \ket_F := \int_{-\pi}^{\pi} f(t) \overline{g(t)} F(dt). 
\]
The correspondence between $\eta \in H(\Xi)$ and $u \in
L^2(F)$ is a unitary isomorphism. 

From \eqref{eq:spec-rep-process}, we see that 
\begin{equation}
 X(z) = \sum_{k=0}^{\infty} 
\left(\int_{-\pi}^{\pi} e^{-ikt} dZ(t)\right) z^k
= \int_{-\pi}^{\pi} \frac{1}{1-ze^{-it}} dZ(t). 
\label{eq:spectral_representation} 
\end{equation}
We have an integral reprsentation of the covariance kernel of the GAF $X_{\Xi}$ from
this expression as follows: 
\begin{equation}
 K(z,w) = \int_{-\pi}^{\pi} \frac{1}{1-z e^{-it}}
\overline{\frac{1}{1-we^{-it}}}F(dt)
\label{eq:Kzw} 
\end{equation}
In particular, for $z= re^{i\phi}$, 
\begin{equation}
 K(z,z) 
= \int_{-\pi}^{\pi} 
\frac{1}{|1-z e^{-it}|^2}F(dt)
=\frac{1}{1-r^2} \tilde{F}(re^{i\phi}) 
\label{eq:Kzz} 
\end{equation}
where $\tilde{F}$ is the harmonic extension of $F(dt)$
defined as the Poisson integral 
\[
\tilde{F}(r e^{i\phi}) 
= \int_{-\pi}^{\pi} P_r(\phi-t) F(dt)
= \int_{-\pi}^{\pi} P_r(s) F_{\phi}(ds)
\]
with $F_{\phi}(ds) = F(d(s+\phi))$ on $(-\pi, \pi]$ modulo $2\pi$. 
Here we give a spectral representation of the $1$-intensity $\rho_1(z)$. 
\begin{prop}\label{prop:1-intensity} 
The $1$-intensity of zeros of $X_{\Xi}(z)$ is given by the formula 
\begin{align}
\disp
\rho_1(z)
&= \frac{1}{\pi(1-r^2)^2}
\frac{\int_{-\pi}^{\pi}\int_{-\pi}^{\pi} \{1 - \cos (t-s)\}
P_r(s)^2P_r(t)^2 F_{\phi}(ds)F_{\phi}(dt)}{\Big( \int_{-\pi}^{\pi} 
P_r(s) F_{\phi}(ds)\Big)^2}, \nonumber 
\end{align}
where $z=re^{i\phi}$ and $F_{\phi}(ds) = F(d(s + \phi))$. 
\end{prop}
\begin{proof}
Edelman-Kostlan's formula \cite{EK95} shows that 
\[
\rho_1(z) 
= \frac{1}{\pi} \partial_z \partial_{\zbar} \log K(z,z) 
= \frac{1}{\pi} \frac{\partial_z \partial_{\zbar}K(z,z) \cdot K(z,z)- \partial_z
K(z,z) \cdot \partial_{\zbar} K(z,z)}{K(z,z)^2}. 
\]
From \eqref{eq:Kzz}, 
\[
K(z,z)^2 = \frac{1}{(1-r^2)^2}
 \tilde{F}(re^{i\phi})^2. 
\]
Also we see that 
\begin{align*}
\lefteqn{\partial_z \partial_{\zbar}K(z,z)\cdot K(z,z) -
 \partial_z K(z,z) \cdot \partial_{\zbar} K(z,z)} \\
&= \frac{1}{2} \int_{-\pi}^{\pi}\int_{-\pi}^{\pi} 
\frac{|e^{it} -
e^{is}|^2}{|1-ze^{-it}|^4|1-ze^{-is}|^4} F(ds)F(dt)\\
&= \frac{1}{(1-r^2)^4} 
\int_{-\pi}^{\pi}\int_{-\pi}^{\pi}
\{1-\cos(t-s)\} P_r(\phi-t)^2P_r(\phi-s)^2 F(ds)F(dt) \\
&= \frac{1}{(1-r^2)^4} 
\int_{-\pi}^{\pi}\int_{-\pi}^{\pi}
\{1-\cos(v-u)\} P_r(u)^2P_r(v)^2 F_{\phi}(du) F_{\phi}(dv). 
\end{align*}
By change of variables $s=\phi+u$ and $t=\phi+v$, we obtain the assertion. 
\end{proof}

For an integrable function $h$, we introduce 
\begin{equation}
 \cP_r(h) := \frac{1}{2\pi} \int_{-\pi}^{\pi}h(s) P_r(s) ds, \quad 
 \cQ_r(h) := \frac{1}{2\pi} \int_{-\pi}^{\pi}h(s) P_r(s)^2 ds. 
\label{eq:PrQr} 
\end{equation}
Suppose $F_{\phi}(ds) = \fphi(s)ds$. Since $P_r(s)$ is
symmetric, we have 
$\cQ_r(\fphi) = \cQ_r(\fphihat)$, 
$\cQ_r(\fphi \cos s) = \cQ_r(\fphihat \cos s)$, and 
$\cQ_r(\fphi \sin s) = \cQ_r(\fphicheck \sin s)$, where
$\fphihat$ and $\fphicheck$ are symmetrization and 
anti-symmetrization of $f_{\phi}$ defined in \eqref{eq:sym-antisym}. 
Therefore, when $F_{\phi}(ds) = \fphi(s)ds$, 
we can restate Proposition~\ref{prop:1-intensity} as
\begin{equation}
 \rho_1(re^{i\phi})
= \frac{1}{\pi(1-r^2)^2} 
\frac{\cS_r(f_{\phi})}{\cP_r(\hat{f}_{\phi})^2}.  
\label{eq:1-intensity}
\end{equation}
where 
\begin{equation}
\cS_r(\fphi) := \cQ_r(\fphihat)^2 - \cQ_r(\fphihat \cos s)^2
 - \cQ_r(\fphicheck \sin s)^2. 
\label{eq:Sr}
\end{equation}
In the next section, we will give asymptotic expansions for $\cP_r(h)$ and 
$\cQ_r(h)$ as $r \to 1$. 

\section{Precise asymptotics of Poisson integral and its relative} 
\label{sec:poisson_integral} 

The following notation will be used in this section and hereafter.
\begin{equation}
 x = x(s) = 1-\cos s, \quad y = y(r) = 1-r^2. 
\end{equation}
Let $\fsp := \{h :(-\pi,\pi] \to \R : \text{$h$ is smooth around $s=0$, and symmetric, i.e. $h(-s)=h(s)$}\}$ and 
$\fbsp := \{h \in \fsp : \text{$h$ is bounded} \}$. 
For $h \in \fsp$, we define 
\[
 Th(s) := \frac{h(s)-h(0)}{x(s)}, \quad 
Th(0) = h''(0).  
\]
Then, $T : \fsp \to \fsp$. 
If, in addition,  $h$ is bounded, so is $Th$; $T$ sends $\fbsp$ to itself. 
We also remark that 
\begin{equation}
 T^2h(0) = (Th)''(0) = \frac{1}{6} (h''(0) + h^{(4)}(0)). 
\label{eq:T2h} 
\end{equation}

\subsection{Asymptotic expansion I}
\label{sec:asymptotics1} 

Fatou's theorem states that if $h$ is integrable on $(-\pi,
\pi]$ and continuous at $s=0$ then $\cP_r(h)$ converges to
$h(0)$ (cf. \cite{Koosis98}). 
The following is a precise version of Fatou's theorem. 
\begin{prop}\label{prop:asymptotics1} 
Suppose $h \in \cH_b$. Then, as $r \to 1$, 
\begin{align}
\cP_r(h)
&=h(0) + \frac{\cI(Th)}{2} y 
+\frac{1}{4} \left\{ \cI(Th) - \frac{h''(0)}{2}\right\} y^2 \nonumber \\
&\quad +\frac{1}{8} \left\{ \frac{3}{2}\cI(Th) - h''(0)
- \frac{1}{2}\cI(T^2h) \right\}y^3 
+O(y^4), 
\label{eq:Prh}
\end{align}
where $y = 1-r^2$. 
\end{prop}
\begin{proof}
 We use Lemma~\ref{lem:Prrecursive} twice to obtain 
\begin{align*}
 \cP_r(h) 
&= h(0) + \frac{y}{2r} \cI(Th) -
 \frac{y^2}{2r(1+r)^2} Th(0)
- \frac{y^3}{4r^2(1+r)^2} \cI(T^2h)\\
&\quad + \frac{y^4}{4r^2(1+r)^4} \cP_r(T^2h). 
\end{align*}
By expanding the functions of $r$ in $y$, we obtain the
 assertion. 
\end{proof}

We observe a recursion relation of the Poisson integral. 
\begin{lem}\label{lem:Prrecursive} 
Suppose $h \in \cH_b$. Then, 
\begin{align*}
\cP_r(h)
=h(0) + \frac{y}{2r} \cI(Th) - \frac{y^2}{2r(1+r)^2} \cP_r(Th). 
\end{align*}
\end{lem}
\begin{proof}
We also define the operator $\cU_r$ by 
\[
\cU_r(f) = \frac{1}{2\pi} \int_{-\pi}^{\pi}
 f(s)\frac{2x}{(1-r)^2 + 2rx} ds. 
\]
Then, 
\[
 \cP_r(h) = h(0) + \frac{1-r^2}{2} \cU_r(Th). 
\]
On the other hand, the following identity 
\[
r \frac{2x}{(1-r)^2 + 2rx} 
= 1 - \frac{1-r^2}{(1+r)^2}P_r(s)
\]
yields 
\[
 r \cU_r(f) = \cI(f) - \frac{1-r^2}{(1+r)^2} \cP_r(f). 
\]
Hence we obtain the equality. 
\end{proof}

\subsection{Asymptotic expansion II}
\label{sec:asymptotics2} 

In this subsection, we give the following asymptotic expansion of $Q_r(h)$ as $r \to 1$. 
Since $P_r(s)ds$ is close to $\delta_0(ds)$, $P_r(s)^2 ds$ is close to $\delta_0(ds)^2$, which gives a diverging term. 
\begin{prop}\label{prop:asymptotics2}
Suppose $h \in \cH_b$. Then, as $r \to 1$, 
\begin{align*} 
 \cQ_r(h) 
&= 2h(0) y^{-1} - h(0) 
+ \frac{1}{4}h''(0) y 
+ \left(\frac{1}{4}\cI(T^2 h) + \frac{1}{8}h''(0) \right)
 y^2 \\
&\quad + \left(\frac{1}{4}\cI(T^2 h) + \frac{1}{16}h''(0) 
- \frac{1}{64} h^{(4)}(0) \right)
 y^3 \\
&\quad + \left(
\frac{1}{4} \cI(T^2 h) -  \frac{1}{16} \cI(T^3 h)
+ \frac{1}{32}h''(0) -\frac{3}{128} h^{(4)}(0) 
\right)
 y^4 \\
&\quad +O(y^5), 
\end{align*}
where $y=1-r^2$. 
\end{prop}
\begin{proof}
Before proving this expansion formula, we first observe
that 
\begin{equation}
 \cQ_r(h) = h(0) \cQ_r(1) + Th(0) \cQ_r(x)
+ \cQ_r(T^2h \cdot x^2) 
\label{eq:Qrh}
\end{equation}
since $h(s) = h(0) + Th(0)x + T^2h(s)x^2$. 
To expand the third term, we introduce two more operators: 
\begin{align*}
\cV_r(h)
&:= \frac{1}{2\pi} \int_{-\pi}^{\pi} h(s)\left\{\frac{2x}{(1-r)^2 + 
 2rx}\right\}^2 ds, \\ 
\cK_r(h) &:= \frac{1}{2\pi} \int_{-\pi}^{\pi} h(s)
\frac{\{2x - (1-r)\}\{(1-r)^2 + 2(1+r)x\}}{\{(1-r)^2 + 2r
 x\}^2} ds.  
\end{align*}
We note that 
\begin{equation}
 \cV_r(h) = \cI(h) + (1-r) \cK_r(h).   
\label{eq:Vrh}
\end{equation}
Since $Q_r(h \cdot x^2) = \frac{y^2}{4} \cV_r(h)$, 
from \eqref{eq:Qrh} and \eqref{eq:Vrh}, we have 
\begin{equation}
 \cQ_r(h) = h(0) \cQ_r(1) + h''(0) \cQ_r(x)
+ \frac{y^2}{4} \cI(T^2h)
+ \frac{y^3}{4(1+r)} \cK_r(T^2h).  
\end{equation}
It is easy to see that 
\begin{align}
  \cQ_r(1) &= \frac{1+r^2}{1-r^2} = \frac{2}{y} - 1, \nonumber\\
  \cQ_r(x) &= \frac{1-r}{1+r} = \frac{1}{4}y + \frac{1}{8} y^2+ \frac{5}{64} y^3 
+ \frac{7}{128} y^4 + O(y^5). 
\label{eq:Qrx}
\end{align}
By \eqref{eq:T2h}, \eqref{eq:expansion4Kr} below, and $(1+r)^{-1} = \frac{1}{2} + \frac{1}{8}y + O(y^2)$, 
we obtain the assertion. 
\end{proof}

We have a recursion formula for $\cK_r$, from which we obtain an expansion formula. 
\begin{lem}\label{lem:recursion4Kr}
Suppose $h \in \cH_b$. 
As $r \to 1$, 
\begin{equation}
\cK_r(h)
=2 \cI(h) -\frac{3}{4} h(0)
+ \left(-\frac{15}{16} h(0) + \frac{3}{2} \cI(h) -
 \frac{1}{2}\cI(Th)
\right) y + O(y^2). 
\label{eq:expansion4Kr} 
\end{equation}
\end{lem}
\begin{proof}
From the identity 
\begin{align*}
\lefteqn{\frac{\{2x - (1-r)\}\{(1-r)^2 + 2(1+r)x\}}{\{(1-r)^2 + 2r x\}^2}} \\
&= 2 + (1-r) 
\left\{
(1+2r) \left(\frac{2x}{(1-r)^2 + 2 rx}\right)^2 
+ (2r-3) \frac{1}{(1-r)^2 + 2 rx}
+ r(2r-3) \frac{2x}{\{(1-r)^2 + 2 rx\}^2}
\right\}, 
\end{align*}
we see that 
\begin{align*}
\cK_r(h)
&= 2 \cI(h) 
+ \frac{1+2r}{1+r} \cV_r(h) y
+ \frac{2r-3}{1+r} P_r(h) 
+ (1-r) r(2r-3) \left\{h(0) \frac{2}{(1+r)^3} \frac{1}{1-r}
 + \frac{1}{2} \cV_r(Th)\right\}. 
\end{align*}
For the fourth term, we used $h(s) = h(0) + xTh(s)$ and
 \eqref{eq:Qrx}. 
By using Lemma~\ref{lem:Prrecursive} and \eqref{eq:Vrh}, 
we obtain the following recursion equation: 
\begin{align}
\cK_r(h) &=2 \cI(h) + \frac{(2r-3)(1+4r+r^2)}{(1+r)^3} h(0) 
+ \left\{\frac{1+2r}{1+r} \cI(h) +
 \frac{(2r-3)(1+r^2)}{2r(1+r)} \cI(Th) \right\} y \nonumber\\
&\quad + \left\{\frac{1+2r}{(1+r)^2} \cK_r(h)  -
 \frac{2r-3}{2r(1+r)^3} \cP_r(Th) + \frac{r(2r-3)}{2(1+r)^2}
\cK_r(Th)\right\} y^2. 
\label{eq:recursion4Kr} 
\end{align}
From \eqref{eq:recursion4Kr} together with 
the Taylor expansions of the coefficients at $r=1$, we see that 
\begin{align*}
\cK_r(h)
&=2 \cI(h) + 
\Big(-\frac{3}{4} -\frac{15}{16} y + O(y^2)\Big) h(0) \\
&\quad + \left\{\Big(\frac{3}{2} - \frac{1}{8} y + O(y^2)\Big) \cI(h) +
 \Big(-\frac{1}{2} -\frac{5}{8}y + O(y^2)\Big) \cI(Th) \right\} y 
+O(y^2), 
\end{align*}
from which we obtain the assertion. 
\end{proof}

\section{Proof of Theorem~\ref{thm:main}}\label{sec:proof}

Suppose $F_{\phi}(ds) = f_{\phi}(s)ds$ and $f_{\phi}(s)>0$. 
From \eqref{eq:1-intensity} and \eqref{eq:Sr}, we recall 
\begin{align*}
\rho_1(re^{i\phi})
&= \frac{1}{\pi(1-r^2)^2} 
\frac{\{\cQ_r(\hat{f}_{\phi})\}^2 - \{\cQ_r(\hat{f}_{\phi} \cos s)\}^2 -
 \{\cQ_r(\check{f_{\phi}} \sin
 s)\}^2}{\{\cP_r(\hat{f}_{\phi})\}^2}, 
\end{align*}
where $\fphihat$ and $\fphicheck$ are symmetrization and
anti-symmetrization of $\fphi$, respectively, as defined in
\eqref{eq:sym-antisym}. 

Suppose $h \in \cH_b$. 
From Proposition~\ref{prop:asymptotics2}, 
setting $y=1-r^2$, we have 
\begin{align*}
 \cQ_r(h) 
&=: \sum_{k=-1}^4 Q_k(h) y^k + O(y^5), 
\end{align*}
where 
\begin{align*}
& Q_{-1}(h) = 2 h(0), \ 
 Q_0(h) = - h(0), \ 
 Q_1(h) = \frac{1}{4} h''(0), \\
& Q_2(h) = \frac{1}{4} \left\{ \cI(T^2 h) + \frac{h''(0)}{2} \right\},
 \\ 
& Q_3(h) = \frac{1}{4}\left\{\cI(T^2 h) 
+ \frac{1}{4} h''(0) - \frac{1}{16} h^{(4)}(0) \right\} \\
& Q_4(h) = 
\frac{1}{4} \left\{\cI(T^2 h) -  \frac{1}{4} \cI(T^3 h)
+ \frac{1}{8}h''(0) -\frac{3}{32} h^{(4)}(0) \right\}. 
\end{align*}
Then, we have 
\begin{align*}
\cR_r(h)
&:= \cQ_r(h)^2 = \sum_{p=-2}^3 \left(\sum_{k+l=p}
Q_k(h)Q_l(h)\right) y^p + O(y^4) \\
&=: \sum_{k=-2}^3 R_k(h)y^k + O(y^4). 
\end{align*} 
A direct computation shows that 
\begin{align*}
 R_{-2}(h) &= 4h(0)^2, \ R_{-1}(h) = -4h(0)^2, \\
 R_0(h) &= h(0) \{h(0) + h''(0)\}, \ R_{1}(h) = h(0) \cI(T^2 h) \\
 R_{2}(h) 
&= h(0) \left\{\frac{1}{2} \cI(T^2 h) - \frac{1}{16} h^{(4)}(0)\right\} 
+ \frac{1}{16} \{h''(0)\}^2 \\ 
&= \frac{1}{2} R_1(h) + \frac{1}{16} \{ h''(0)^2 - h(0) h^{(4)}(0) \} \\
 R_3(h) 
&= \frac{1}{2}\left\{h(0)+\frac{1}{4} h''(0)
 \right\} \cI(T^2 h) - \frac{1}{4} h(0) \cI(T^3 h) + \frac{1}{16} \{h''(0)^2 -
 h(0) h^{(4)}(0)\} \\ 
&= R_2(h) + \frac{1}{8}h''(0) \cI(T^2 h) - \frac{1}{4} h(0)
 \cI(T^3 h). 
\end{align*}

For $g$ (instead of $f_{\phi}$) which is not necessarily symmetric, we now compute
\eqref{eq:Sr}, i.e., 
\begin{equation}
\cS_r(g) := \cR_r(\ghat) - \cR_r(\ghat \cos s) - \cR_r(\gcheck \sin s) 
= \sum_{k=-2}^3 S_k(g) y^k + O(y^4),  
\label{eq:Srg}
\end{equation}
which is equivalent to 
\[
 S_k(g) = R_k(\ghat) - R_k(\ghat \cos s) - R_k(\gcheck \sin
 s). 
\]
We note that if $h \in \cH$, then 
\[
(h \cos s) ''(0) = h''(0) - h(0), \  
h^{(4)}(0) - (h \cos s)^{(4)}(0) = -h(0) + 6 h''(0). 
\]
and 
\[
T h - T(h \cos s) = h
\]
Then we can compute $S_k(g)$ as follows: 
\begin{align*}
S_{-2}(g)
&=S_{-1}(g)= 0, \ S_0(g) = \ghat(0)^2 \\
S_1(g)
&= \ghat(0) \{\cI(T^2 \ghat - T^2(\ghat \cos s))\} = \ghat(0) \cI(T\ghat)\\ 
S_2(g)
&= \frac{1}{2}\ghat(0) \cI(T\ghat) 
+ \frac{1}{16} \{\ghat''(0)^2 -(\ghat \cos s)''(0)^2
-(\gcheck \sin s)''(0)^2\} \\
&\quad - \frac{1}{16} 
\ghat(0) \{\ghat^{(4)}(0) - (\ghat \cos s)^{(4)}(0)\} \\
&= \frac{1}{2}\ghat(0) \cI(T\ghat) - 
\frac{1}{4} \Big( \ghat(0) \ghat''(0) + \gcheck'(0)^2 \Big) \\
S_3(g) 
&=S_2(g) 
+ \frac{1}{8} \left\{\ghat''(0) \cI(T\ghat) 
+ \ghat(0) \cI(T^2 \ghat \cos s) -2 \gcheck'(0) \cI(T^2\gcheck
 \sin s)\right\} - \frac{1}{4} \ghat(0)\cI(T^2 \ghat). 
\end{align*}
When $g(0) >0$, 
\begin{equation}
\cS_r(g) = g(0)^2 + g(0) \cI(T\ghat) y 
+ \left\{\frac{1}{2} g(0) \cI(T\ghat) - 
\frac{1}{4} \Big( g(0) g''(0) + g'(0)^2 \Big)\right\} y^2 + O(y^3). 
\label{eq:Srnonzero}
\end{equation}
When $g(0) =0$, $g'(0)=0$, we have 
\begin{equation}
\cS_r(g) = \frac{1}{8} g''(0)\cI(T\ghat) y^3 + O(y^4).  
\label{eq:Srzero}
\end{equation}

\subsection{The case $f_{\phi}(0)>0$}
Suppose the spectral density at $\phi$ is 
strictly positive, i.e., $f_{\phi}(0) > 0$. 
From \eqref{eq:1-intensity}, \eqref{eq:Prh}, 
and \eqref{eq:Srnonzero}, we see that 
\begin{align*}
\rho_1(re^{i\phi})
&= \frac{1}{\pi(1-r^2)^2} 
\frac{\cS_r(f_{\phi})}{\cP_r(\hat{f}_{\phi})^2} \\
&= \frac{1}{\pi(1-r^2)^2} 
\frac{f_{\phi}(0)^2 + f_{\phi}(0) \cI(T\fhat_{\phi}) y +
 b_{\phi} y^2
 +O(y^3)}{f_{\phi}(0)^2 + f_{\phi}(0) \cI(T
 \fhat_{\phi}) y + c_{\phi} y^2 + O(y^3)} \\
&=\frac{1}{\pi(1-r^2)^2} 
\left\{1 + \frac{b_{\phi}-c_{\phi}}{f_{\phi}(0)^2} y^2 +
 O(y^3)\right\}  \\   
&=\frac{1}{\pi(1-r^2)^2} 
\left\{1 - \frac{1}{4f_{\phi}(0)^2} \{f_{\phi}'(0)^2 +
 \{\cI(T \fhat_{\phi})\}^2 \}y^2 + O(y^3)\right\} .   
\end{align*}
We used the expansion formula 
\[
 \frac{p+au+bu^2+O(u^3)}{p+au+cu^2+O(u^3)}
= 1 + \frac{b-c}{p} u^2 + O(u^3) \quad \text{ as $u \to 0$}. 
\]

\subsection{The case $f_{\phi}(0)=0$}
First, note that when $f_{\phi}(0)=0$, we have $f'_{\phi}(0)=0$
since $f_{\phi}(s) \ge 0$. 
We have the following asymptotics from Propositions~\ref{prop:asymptotics1} and \ref{prop:asymptotics2}
\begin{align*}
\cP_r(\fphihat)
&=\frac{\cI(T\fphihat)}{2} y 
+\frac{1}{4} \left( \cI(T\fphihat) - \frac{\fphi''(0)}{2}\right) y^2 
+ O(y^3),  
\end{align*}
and thus 
\begin{equation}
\cP_r(\fphihat)^2
=\frac{1}{4}\cI(T\fphihat)^2 y^2 
+ \frac{1}{4} \cI(T\fphihat)
\left( \cI(T\fphihat) - \frac{\fphi''(0)}{2}\right) y^3+
 O(y^4).
\label{eq:Prsquare}
\end{equation}
On the other hand, we see that 
\begin{equation}
 \cS_r(\fphi) 
= \frac{1}{8} \fphi''(0) \cI(T\fphihat)y^3 + O(y^4). 
\label{eq:Sr2}
\end{equation}
Therefore, from \eqref{eq:Prsquare} and \eqref{eq:Sr2}, we
obtain 
\[
\frac{\cS_r(\fphi)}{\cP_r(\fphi)^2}
= \frac{\fphi''(0)}{2\cI(T\fphihat)}y + O(y^2). 
\]
This implies Theorem~\ref{thm:main}(ii). 

\subsection{The case $\fphi(0)=\fphi''(0)=0$}

When $\fphi(0)=\fphi''(0)=0$ (and thus $\fphi'(0)=0$), 
from Proposition~\ref{prop:asymptotics1}, 
we have 
\begin{equation*}
\cP_r(\fphihat) 
=\frac{1}{2} \cI(T\fphihat)y
+\frac{1}{4} \cI(T\fphihat) y^2 + O(y^3), 
\end{equation*}
and hence 
\begin{equation}
\cP_r(\fphihat)^2
= \frac{1}{4} \cI(T\fphihat)^2 y^2 + O(y^3).
\label{eq:Prsquare2}
\end{equation}
From Proposition~\ref{prop:asymptotics2}, we have  
\[
 \cQ_r(\fphihat) = \frac{1}{4}\cI(T^2 \fphihat) y^2 +
 O(y^3), 
\]
and thus   
\begin{align}
 \cS_r(\fphi) 
&= \cQ_r(\fphihat)^2 - \cQ_r(\fphihat \cos s)^2  -
 \cQ_r(\fphicheck \sin s)^2 \nonumber \\
&= \frac{1}{16}
\left\{\cI(T^2 \fphihat)^2 - \cI(T^2 \fphihat \cos s)^2 
- \cI(T^2 \fphicheck \sin s)^2 \right\} y^4 + O(y^5)
\label{eq:Sr3}
\end{align}
Therefore, from \eqref{eq:Prsquare2} and \eqref{eq:Sr3}, we
obtain 
\begin{align*}
\rho_1(re^{i\phi})
&=\frac{1}{\pi(1-r^2)^2} \frac{\cS_r(\fphi)}{\cP_r(\fphi)^2}\\
&= \frac{1}{\pi} 
\frac{\cI\Big( T^2 \fphihat\Big)^2 - \cI\Big(T^2(\fphihat \cos s)\Big)^2 -
 \cI\Big( T^2(\fphicheck \sin s) \Big)^2} 
{4 \cI(T\fphihat)^2} + O(y). 
\end{align*}

\begin{rem}
The numerator of the first term above can be expressed as 
\[
 \frac{1}{4} \int_{-\pi}^{\pi}\int_{-\pi}^{\pi} 
\frac{G_{\phi}(s,t) + G_{\phi}(s,-t)}{(1-\cos s)^2 (1-\cos
 t)^2} dsdt, 
\]
where 
$G_{\phi}(s,t) := \{f_{\phi}(s)f_{\phi}(t)+f_{\phi}(-s)f_{\phi}(-t)\}\{1-\cos(s-t)\}$. 
\end{rem}

\section{Analytic continuation}\label{sec:analytic_continuation} 

First we give a proof of the following basic fact. 
\begin{lem}\label{lem:normal_sequence} 
Let $\{\zeta_n\}_{n=0}^{\infty}$ be a sequence of standard complex
normal random variables that are not necessarily independent. Then, 
\begin{equation}
 \lim_{n \to \infty} |\zeta_n|^{1/n} = 1 \text{ a.s.} 
\label{eq:limzetan} 
\end{equation}
\end{lem}
\begin{proof}
Since $\P(|\zeta| \ge r) = e^{-r^2}$ for $\zeta \sim N_{\C}(0,1)$,  
we have 
\[
 \sum_{n =1}^{\infty} \P(|\zeta_n| \ge \sqrt{2 \log n}) \le 
 \sum_{n =1}^{\infty} \frac{1}{n^2} < \infty 
\]
and 
\[
 \sum_{n =1}^{\infty} \P(|\zeta_n| < \frac{1}{n}) 
= \sum_{n =1}^{\infty} (1 - e^{-\frac{1}{n^2}}) 
\le  \sum_{n =1}^{\infty} \frac{1}{n^2} < \infty, 
\]
which imply 
\[
\P\Big(\frac{1}{n} \le |\zeta_n| \le \sqrt{2\log n} \text{ f.e.}
 \Big) =1. 
\]
Therefore, we obtain the assertion.  
\end{proof}

Suppose $h(z)$ is a power series of convergence radius $\rho \in (0,\infty)$. We recall that a point $b$ on $|z|=\rho$ is singular (or a singular point) if there does not exist any analytic continuation of $h(z)$ across an arc of $\{|z|=\rho\}$ containing $b$; otherwise $b$ is called regular (or a regular point). 
A subset of an arc $|z|=\rho$ is regular if it consists of regular points. 
If there exist $0<r<1$ and $0 < \eta < \pi$ such that 
\begin{equation}
\rho(r):= \left(\limsup_{k \to \infty} \left|\frac{1}{k!} X_{\Xi}^{(k)}(r)
\right|^{1/k}\right)^{-1} > |e^{i\eta} -r| = (1 - 2r \cos
\eta +r^2)^{1/2}, 
\label{eq:ac_event} 
\end{equation}
then the arc $I_{\eta} = \{e^{is} : |s| \le \eta \} \subset \partial
\D$ is regular (cf. \cite[p.39]{Kahane85}) 
since $\rho(r)$ is the radius of convergence centered at $r$
and the disk with radius $\rho(r)$ centered at $r$ contains
the arc $I_{\eta}$. 
Using this fact, we provide a proof of Theorem~\ref{thm:analytic_continuation}. 

\begin{proof}[Proof of Theorem~\ref{thm:analytic_continuation}]
Since the open set $(\supp F)^c$ can be expressed as at most 
 a countable union of open arcs, 
we regard each open arc as $I_{\delta}^{\circ} := \{e^{is} :
 s\in (-\delta, \delta)\}$ for some $\delta>0$ by applying an
 appropriate rotation if needed. 
We check \eqref{eq:ac_event} to show the open arc 
$I_{\delta}^{\circ}$ is regular.  
From the spectral represenation
 \eqref{eq:spectral_representation},  
we see that 
\[
\alpha_k = \alpha_k(z) := \frac{1}{k!} X_{\Xi}^{(k)}(z) = \int_{-\pi}^{\pi}
 \frac{e^{-ikt}}{(1-ze^{-it})^{k+1}} dZ(t). 
\]
Therefore, 
for $z=r \in (0,1)$, 
\begin{align*}
\var(\alpha_k) 
&=\E\left[\left|\frac{1}{k!} X_{\Xi}^{(k)}(z) \right|^2\right] \\
&= \int_{-\pi}^{\pi}
 \frac{1}{|1-re^{-it}|^{2(k+1)}} F(dt) \\
&= \int_{-\pi}^{\pi} \frac{1}{(1-2r \cos t + r^2)^{k+1}}
 F(dt). 
\end{align*}
Since $\{\alpha_k\}$ is a sequence of centered Gaussian
 random variables and 
$\alpha_k = \var(\alpha_k)^{1/2} \zeta_k$ with $\zeta_k :=
 \alpha_k/\var(\alpha_k)^{1/2} \sim N_{\C}(0,1)$, 
by Lemma~\ref{lem:normal_sequence}, we have 
\[
\rho(r) = \left(\limsup_{k \to \infty}
 |\alpha_k|^{1/k}\right)^{-1} 
= \limsup_{k \to \infty} \var(\alpha_k)^{-1/2k} \quad \as 
\]

Since the function $1-2r \cos t + r^2$ is increasing in $t
 \in (0,\pi)$ for $r \in (0,1)$, and $\supp F \subset
 (I_{\delta}^{\circ})^c$, we have 
\begin{align*}
\rho(r) 
&= \limsup_{k \to \infty} 
\left(\int_{-\pi}^{\pi} \frac{1}{(1-2r \cos t + r^2)^{k+1}}
 F(dt) \right)^{-1/2k}\\
&\ge (1-2r \cos \delta + r^2)^{1/2}. 
\end{align*}
From \eqref{eq:ac_event}, the arc $I_{\eta}$ is regular for
 any $\eta \in (0,\delta)$, and thus $I_{\delta}^{\circ}$ is 
 regular. 

On the other hand, when $0 \in \supp F$, for every $\epsilon \in
 (0,\pi)$, we see that 
\begin{align*}
 \rho(r) 
&\le \limsup_{k \to \infty} 
\left(\int_{-\epsilon}^{\epsilon} \frac{1}{(1-2r \cos t + r^2)^{k+1}}
 F(dt) \right)^{-1/2k} \\
&\le \limsup_{k \to \infty} 
\left(F([-\epsilon, \epsilon])\cdot (1-2r \cos \epsilon +
 r^2)^{-(k+1)} \right)^{-1/2k} \\
&= (1-2r \cos \epsilon + r^2)^{1/2}. 
\end{align*}
Since $\epsilon>0$ is arbitrary, $\rho(r) \le 1-r$ for every
 $r \in (0,1)$, which implies that $X_{\Xi}(z)$ cannot be
analytically extended across $z=1$, i.e., $z=1$ is a singular point. 
Therefore, we can conclude that $(\supp F)^c$ is the regular
 set of $X_{\Xi}$. 
\end{proof}

Although we provided a direct proof of 
the fact that the singularity of $X_{\Xi}$ coincides with
$\supp F$, 
the spectrum of $\Xi$, we here present an alternative proof 
which was essentially given in \cite{ANS17} using the theory of entire
functions (cf. \cite{Levin80, Levin96}). 

For a set $A \subset \C$, we denote the closed convex hull of $A$ by
$A^{\conv}$ and the reflection of $A$ in the real axis by
$A^{\refl} := \{\zbar : z \in A\}$. 

We consider an entire function of exponential
type $\tau$ of the form 
\[
f(z) = \sum_{n=0}^{\infty} a_n\frac{z^n}{n!}, 
\]
where type $\tau$ implies that 
\begin{equation}
 \limsup_{n\to \infty} |a_n|^{1/n} = \tau. 
\label{eq:type-sigma}
\end{equation}
One can associate such entire function $f(z)$ with its Borel transform 
\[
 \phi_f(w) = \sum_{n=0}^{\infty} \frac{a_n}{w^{n+1}}. 
\]
From \eqref{eq:type-sigma}, $\phi_f(w)$ is holomorphic in
the domain $D^{\tau} = \{w \in \C :|w| > \tau\}$. Now we denote
the singularity of $\phi_f$ by $S^{\phi_f} \subset \partial
D^{\tau} = \{w \in \C : |w|=\tau\}$. The function $z^{-1}
\phi_f(z^{-1}) = \sum_{n=0}^{\infty} a_n z^n$ is holomorphic
in $\D_{\tau}$ and its singularity coincides with $S^{\phi_f}$. 
The closed convex hull $(S^{\phi_f})^{\conv}$ is called the
\textit{conjugate diagram} of the function $f(z)$. 
We define the \textit{indicator function} of $f(z)$ by 
\[
 h^f(\theta) := \limsup_{n \to \infty}
 \frac{\log|f(re^{i\theta})|}{r} \quad \text{for $\theta \in [-\pi, \pi]$}. 
\]
It is known that $h^f(\theta)$ is the supporting function of
a bounded convex set, which is called the \textit{indicator
diagram} of $f(z)$ and denote it by $I^f$ \cite[Chapter I, Section 19]{Levin80}.

Now we recall Poly\'{a}'s theorem \cite[Theorem 33, Chapter
I]{Levin80}. 
\begin{thm}
 The conjugate diagram of an eitire function of exponential
 type is the reflection in the real axis of its indicator
 diagram, i.e., $(S^{\phi_f})^{\conv} = (I^f)^{\refl}$.  
\end{thm}

In \cite[Lemma~7.2.1]{ANS17}, 
for a wide-sense stationary process $\Xi=\{\xi_n\}_{n \in
\Z}$ with spectral measure $F(dt)$, they considered an entire
function of exponential type of the form 
\[
E_{\Xi}(z) = \sum_{n=0}^{\infty} \xi_n\frac{z^n}{n!}
\]
and showed that 
almost surely,
\begin{equation}
 h^{E_{\Xi}}(\theta) \le h_{F^*}(\theta), 
\quad \theta \in [-\pi, \pi], 
\label{eq:hEgehF} 
\end{equation}
where $h_{F^*}(\theta)$ is the supporting function of 
$((\supp F)^{\refl})^{\conv}$ and given by 
$h_{F^*}(\theta) = \max_{t \in \supp F} \cos (\theta + t)$. 
In other words, the indicator diagram $I^{E_{\Xi}}$
 of $E_{\Xi}$ is contained in the closed
convex hull of $\supp F$ reflected in the real axis, i.e., 
\begin{equation}
(S^{\phi_{E_{\Xi}}})^{\conv} = (I^{E_{\Xi}})^{\refl} \subset
(\supp F)^{\conv}, 
\label{eq:Sf=If} 
\end{equation}
and hence, we can conclude that $S^{\phi_{E_{\Xi}}} \subset
\supp F$ since both $S^{\phi_{E_{\Xi}}}$ and $\supp F$ are subsets of
$\partial D^{\tau}$. This implies that analytic continuation
can be performed at least across $(\supp F)^c$. 
Furthermore, when $\Xi$ is a stationary
Gaussian process, they showed \cite[Thereom 4]{ANS17} that almost surely 
\[
 h^{E_{\Xi}}(\theta) = h_{F^*}(\theta) \quad \theta \in [-\pi,
 \pi]. 
\]
Therefore, in the case where $\Xi$ is a stationary Gaussian process, we have 
$S^{\phi_{E_{\Xi}}} = \supp F$. 

Understanding the details of the analytic
continuation of $X_{\Xi}(z)$ (or equivalently
$\phi_{E_{\Xi}}(w)$) 
for a wide-sense stationary process (or an even Wiener sequence) $\Xi$ provides information on the growth behavior of
$E_{\Xi}(z)$ or $\Xi$ itself (see also \cite{BSW18, BBS21}). 


\vskip 5mm
\textbf{Acknowledgment. }
The author would like to thank Misha Sodin for pointing
out an alternative proof of Theorem~\ref{thm:analytic_continuation}. 
The author would also like to thank Naomi Feldheim for sharing
some relevant references. 
This work was supported by JSPS KAKENHI Grant Numbers
JP22H05105, JP23H01077 and JP23K25774, and also supported in
part JP21H04432 and JP24KK0060.  

\bibliographystyle{plain}
\bibliography{referenceGAF} 

\end{document}